\documentclass[a4paper,12pt]{amsart}

\usepackage[T1]{fontenc}
\usepackage{amsmath,amssymb,amscd,amsfonts,amsthm,mathrsfs,eucal,dsfont,verbatim}

\usepackage{lscape,enumerate}
\usepackage{xcolor,soul}

\usepackage[margin=2in]{geometry}
\usepackage{hyperref}
\usepackage{marginnote}

\numberwithin{equation}{section}

\newcommand{\real}{\mathds{R}}

\newcommand{\integer}{\mathds Z}
\newcommand{\Ee}{\mathds E}
\newcommand{\Pp}{\mathds P}
\newcommand{\I}{\mathds 1}

\newcommand{\Ff}{\mathcal{F}}

\newcommand{\LL}{\mathcal{L}}
\newcommand{\Bb}{\mathcal{B}}
\def\1{1\!\!\hbox{{\rm I}}}

\newcommand{\rd}{{\mathds{R}^d}}

\newcommand{\Aa}{\mathcal{A}}

\usepackage{color}

\evensidemargin
-1cm

\usepackage[T2A]{fontenc}
\usepackage[english]{babel}

\numberwithin{equation}{section}
\theoremstyle{plain}
\newtheorem{theorem}{Theorem}
\newtheorem{lemma}{Lemma}

\theoremstyle{definition}
\newtheorem{remark}{Remark}

\newtheorem{example}{Example}

\numberwithin{theorem}{section}
\numberwithin{definition}{section}
\numberwithin{corollary}{section}
\numberwithin{remark}{section}
\numberwithin{lemma}{section}

\begin{document}

\title{On ergodic property of some  L\'evy-type processes in $\rd$}

\date{}

\author[Y. Mokanu]{Yana Mokanu}
\address[Y. Mokanu]{I. Sikorsky Kyiv Polytechnic Institute\\ Dep. of Mathematical Analysis and Probability Theory\\ Beresteyskyi Ave. 37, 03056 Kyiv, Ukraine\\
	\texttt{Email:  $yana.mokanu3075$@$gmail.com$} }

\begin{abstract}
	In this paper, we investigate the ergodicity in total variation of the process $X_t$ related to some integro-differential operator with unbounded coefficients and describe the speed of convergence to the respective invariant measure. Some examples are provided. 
\end{abstract}


\maketitle

\section*{Introduction} 
Consider an integro-differential operator 
\begin{equation}\label{L01}
	\begin{split}
	Lf(x) &=  \ell(x)\cdot \nabla f(x)
	+ \int_{\rd \backslash \{0\}} \left(f(x+u) - f(x) -\nabla f(x) \cdot u \I_{|u|< 1} \right) \nu(x,du), 
	\end{split}
\end{equation}
defined on the set $C_0^2(\rd)$ of twice continuously differentiable functions with compact support, where  $\ell(\cdot): \rd \to \rd$, and  the  Lévy-type  kernel  $\nu(x,du)$ satisfies 
\begin{equation*}\label{ubdd}
\int_{\rd\backslash \{0\}} (1\wedge |u|^2) \nu(x,du)<\infty \quad \text{for any $x\in \rd$}. 
\end{equation*}
To simplify the notation, we assume that the kernel $\nu(x,du)$ does not have an atom at 0 and write $\int_\rd$ instead of $\int_{\rd\backslash \{0\}}$. 

Under certain conditions (see Section~\ref{Prel}  below) one can associate with $(L, C_0^2(\rd))$  a  Markov process $X$, whose transition  probability kernel $P_t(x,dy)$  possesses certain regularity properties.  The  process $X$ is often called a \emph{L\'evy-type} process, because locally its structure resembles the structure of a L\'evy process, e.g. a  process with independent  stationary  increments.

The aim of this paper is to derive sufficient conditions for ergodicity of the process $X$ related to \eqref{L01} and find the ergodic rates. 

This work further develops our previous studies on the ergodicity of Lévy-type processes: we began our investigation in \cite{KM24} with a one-dimensional case and bounded Lévy-type kernel; continued in  \cite{KM25},  where a Lévy-driven stochastic differential equation (SDE) with unbounded coefficients was considered in dimension $d$.

The long-time behaviour of Markov processes has been extensively investigated. Skorokhod in \cite{Sk87} discussed a long-time behaviour of general Markov processes and solutions to SDEs, as well as stabilisation properties of solutions to systems of SDEs. Masuda in \cite{Ma04}, \cite{Ma07} investigated  a L\'evy-driven Ornstein-Uhlenbeck process, in particular, the existence of a smooth density, ergodicity and convergence rates.  Sufficient conditions for recurrence, positive recurrence and exponential ergodicity of one-dimensional L\'evy-type processes generated by the integro-differential operator were derived in \cite{Wa08}. In \cite{Wa13} sufficient conditions for exponential and strong ergodicity of stochastic differential equations driven by symmetric $\alpha-$stable processes were obtained.  Ergodicity of a solution to the jump-diffusion SDE  was studied by Wee in \cite{We99}. In \cite{FR05}  and \cite{DFG09}  sufficient conditions for ergodicity were given in the form of so-called $(f,r)$-modulated moments.  In \cite{Ku15} the problem of the uniqueness of an invariant probability measure of a Markov process was investigated; here the authors also derived sufficient conditions for the weak convergence of the transition probability to the invariant measure.  Kulik in  \cite{Ku18} discussed the asymptotic behaviour of a wide range of Markov processes from simple Markov chains with finite state spaces to more sophisticated models  (including solutions to L\'evy driven SDEs), in particular, the author examined  different types of ergodic rates  and their further application in limit theorems;  the question of the weak ergodicity was addressed as well. Sandri\'c in \cite{Sa13} gave sufficient conditions for recurrence, transience and ergodicity of stable-like processes in terms of their coefficients. Further, sufficient conditions for recurrence, transience and ergodicity of stable-like Markov chains under certain assumptions on their density functions and drift were derived in \cite{Sa14}. In \cite{Sa16} the author discussed recurrence, transience, ergodicity, strong, (sub)-exponential ergodicity and mixing properties of L\'evy-type processes driven by the infinitesimal generator. 

The structure of this paper is as follows.  In Section~\ref{Prel},  we recall the notion of ergodicity and  discuss the assumptions.  In Section~\ref{Set} we formulate our results. Proofs are given in Section~\ref{Proofs} and examples in Section~\ref{Examples}.

\section{Preliminaries }\label{Prel}

A Markov process $(X_t)_{t\geq 0}$  on $\rd$  with a transition probability kernel $P_t(x,dy)$ is called (strongly)
\emph{ergodic} (see \cite{Sa16}) if  there exists an  invariant probability measure $\pi(\cdot)$ such that 
\begin{equation*}\label{inv1}
	\lim_{t\to \infty} \| P_t(x,\cdot)- \pi(\cdot)\|_{TV} =0
\end{equation*}
for any $x\in \real^d$.
Here $\|\cdot\|_{TV}$ is the total variation norm. Moreover, one can describe the speed of convergence to the invariant measure, if certain assumptions are satisfied. The process is called  (non-uniformly) $\psi$-ergodic  (see \cite{Sa16} for polynomial and exponential rates), if there exists a function $\psi\geq 0$, $\psi(t) \to 0$ as $t\to\infty$, such that 
\begin{equation}\label{psierg}
	 \|P_t(x,\cdot )- \pi(\cdot)\|_{TV}\leq C(x) \psi(t), \quad t\geq 1, 
\end{equation}
 for some $C \colon \rd \to [0,\infty)$. 

There are several possible ways to show ergodicity of a Markov process (see \cite[\S3.5]{Ku18}). We use the skeleton chain approach, described in \cite{Ku18}, which allows us to reduce the continuous-time setting to the discrete one by constructing a skeleton chain $X^h :=\{X_{nh}, n \in \integer_+\}$ and apply ergodic theorems, developed specifically for Markov chains. Under certain irreducibility assumptions (e.g., Dobrushin condition), ergodicity of the skeleton chain implies ergodicity of the initial process with the same convergence rate (cf. \cite[\S3.2]{Ku18}, \cite{Sa16}).

To see if the chain is ergodic, one can check the (local) Dobrushin (or other forms of irreducibility condition, suitable for a particular task (see the discussion in \cite[p.64]{Ku18})) and the Lyapunov-type conditions (\cite[Th.2.8.9, Th.3.2.3]{Ku18}).  Note that the irreducibility-type assumption is crucial for the described approach, but if it fails, it is still possible for the process to be ergodic only in a weaker sense (see the discussion in \cite[\S4]{Ku18}).

A  Markov chain $X$ with a transition kernel $\Pp(x,dy)$  satisfies the \textit{local Dobrushin} condition (cf. \cite[p. 42]{Ku18}) on a measurable set $B\subset \rd\times \rd$  if 
\begin{equation}\label{Dor1}
	\sup_{x,y\in B} \|\Pp(x,\cdot)- \Pp(y,\cdot)\|_{TV} <2. 
\end{equation}

In practice, it is sufficient and reasonable to check if there exists an absolutely continuous component of $\Pp(x,dy)$, such that $x\mapsto \Pp(x,\cdot)$ is continuous in $L_1$ (see \cite[Prop.2.9.1]{Ku18}), instead of proving \eqref{Dor1} directly. This, in turn, is the case if the initial Markov process possesses a strictly positive transition probability density $p_t(x,\cdot)$,  such that $x\mapsto p_t(x,\cdot)$ is continuous in $x$ in $L_1$  (see \cite{Sa16}). In other words, it means that the respective semigroup $(P_t)_{t\geq 0}$, $P_t f(x):= \Ee^x f(X_t)$, possesses the $C_b$-Feller property,  that is, it preserves the space of bounded continuous functions (see \cite{BSW13}).  The Dobrushin condition ensures that the process converges to the same invariant measure, regardless of its starting point. 

The \textit{Lyapunov-type} condition (see \cite[p.52]{Ku18}) allows us to establish moment estimates, responsible for (non-uniform) ergodic rates. Suppose that there exist:
\begin{itemize}
	\item   a  norm-like function $V: \, \rd\to [1,\infty)$ (i.e. $V(x)\to \infty$ as $|x|\to  \infty$),   bounded on a compact $K$; 
	
	\item  a function $f: \, [1,\infty) \to (0,\infty)$, that  admits a non-negative increasing and concave extension to $[0,\infty)$;
	
	\item  a constant $C>0$,
\end{itemize}
such that the following  relations  hold:
\begin{equation}\label{Lyap1}
	\Ee^x V(X_1) - V(x) \leq - f (V(x))+ C, \quad x\in \rd, 
\end{equation}
\begin{equation*}
	f\left(1+ \inf_{x\notin K} V(x)\right) > 2C.  
\end{equation*}
Inequality \eqref{Lyap1} is called the \emph{Lyapunov inequality}, and such  a function $V$ is called a \emph{Lyapunov function}.  

Given that we are in the continuous-time settings, it is reasonable to consider the Lyapunov-type condition in terms of the process $X$ itself, namely, in terms of its \emph{full} generator $\LL$:

\begin{equation}\label{Lyap2}
	\LL  V(x) \leq - f(V(x))  + C,
\end{equation}
 where $f, V$ and $C$ have the same meaning as above. 

The full generator is taken into consideration because Lyapunov functions are often unbounded and the domain of $L$ is too narrow to include them. Recall the ``stochastic version'' of a \emph{full generator}, see \cite[(1.50)]{BSW13}. The full generator of a measurable contraction semigroup $(P_t)_{t\geq 0}$  on  the space $C_\infty(\real^d)$ of continuous functions vanishing at infinity  is 
    \begin{gather*}
        \hat{L}:= \Big\{ (f,g)\in(\Bb(\rd),\Bb(\rd)) :\quad   M_t^{[f,g]}:=f(X_t)-f(x)= \int_0^t g(X_s)ds,\quad x\in \rd,\, t\geq 0,\Big.\\
		\Big.\text{is a local  $\Ff_t^X:= \sigma(X_s, s\leq t)$ martingale} \Big\}.
    \end{gather*}
Note that in general $\hat{L}$ is not single-valued, see the discussion in \cite[p.36]{BSW13}. 

Since  for $f\in C_0^2(\rd)$  
\begin{equation}\label{gen1}
	P_t f(x) -f(x) = \int_0^t P_s Lf(x) ds \quad \forall t\geq 0, \,x\in \rd,
\end{equation}
see \cite[Ch.1, Prop.1.5]{EK86},  by the strong Markov property and the Dynkin formula $(f,Lf)\subset \hat{L}$ for any  $f\in C_0^2(\rd)$.
Further,   it is not difficult to check that 
 $Lf(x)<\infty$ for $f\in \mathcal{D}_0$,  where 
\begin{equation*}
	\mathcal{D}_0:=\left\{ f\in C^2(\rd): \quad\left| \int_{|u|\geq 1} (f(x+u)-f(x))\nu(x,du)\right|<\infty \quad \forall x\in \rd \right\}. 
\end{equation*}
Using the localization procedure, the Ito formula  and taking the expectation,  we derive  \eqref{gen1} for any $f\in \mathcal{D}_0$ and, consequently, we get that   $(f,Lf)\in \hat{L}, \,  f\in \mathcal{D}_0$.  
We denote this extension of $L$ to $\mathcal{D}_0$ by $\LL$.

Summarising, condition \eqref{Lyap2} for the initial process implies \eqref{Lyap1} for its skeleton chain (see \cite[Th.3.2.3]{Ku18}), which along with the Dorbrushin condition ensures the ergodicity  in total variation  of this chain and provides the (non-uniform) ergodic rate (cf. \cite[Th.2.8.6, Th.2.8.8.]{Ku18}). As it was mentioned above, the ergodicity of the chain implies the ergodicity of the initial process with the same ergodic rate.

The speed of  convergence to $\pi$  in \eqref{psierg} as $t\to \infty$ is (cf. \cite[Th.2.8.8, Th.3.2.3]{Ku18})
\begin{equation}\label{psi}
	\psi (t) := \left(\frac{1}{f(F^{-1}(\gamma t))}\right)^\delta, \qquad F(t) := \int\limits_1^t \frac{dw}{f(w)}, \quad t \in (1, \infty),
\end{equation}
where $\gamma,\delta\in (0,1)$. Moreover, 
\begin{equation*}
	\int_{\rd} f(V(x)) \pi (dx) < \infty. 
\end{equation*}

Our standing assumptions in this paper  would be 

\begin{align}\label{A1}\tag{\bfseries{A1}}
	&\text{There exists a   solution $\{X_t, \,t>0\}$ to the martingale problem for $(L, C_0^2(\rd))$; }\\ \label{A2}\tag{\bfseries{A2}}
	&
	\text{The  skeleton chain  $X^h := \{ X_{nh}, \, n\in \mathbb{Z}_+\}$ satisfies  the Dobrushin condition.}	
\end{align}

Condition \eqref{A1} assures the existence of the process associated with \eqref{L01} (which we call the \emph{Markovian selection} related to $L$), but it looks rather  abstract and is not easy to check. The sufficient conditions ensuring \eqref{A1} were obtained in \cite{Ku20} and are the following:
\begin{align}\label{A11}\tag{\bfseries{A1'}}
	&\text{The symbol  $q$ is locally bounded,  continuous in $x$, satisfies $q(x, 0) = 0$;  }\\ \label{A111}\tag{\bfseries{A1''}}
	&
	\lim_{R\to \infty}  \sup_{|x|\leq R}\sup_{|\xi|\leq R^{-1}} |q(x,\xi)| =0.
\end{align}
Recall that the symbol $q(x,\xi)$ of the operator $L$ is a function, such that for $f$ from the Schwartz space $S(\rd)$  
$$
	L f(x) = - \int_\rd  e^{i\xi \cdot x} q(x,\xi)\hat{f}(\xi)d\xi, 
$$
where $\hat{f}$ is the Fourier transform of $f$, and
\begin{align*}
	q(x,\xi)&= - i\ell(x)\cdot \xi  
	+ \int_\rd \left(1-e^{i\xi \cdot x} + i \xi \cdot u \I_{|u|\leq 1}  \right) \nu(x,du). 
\end{align*}

Under \eqref{A11} and \eqref{A111}, there exists a solution to the Martingale problem for $(L, C_0^\infty(\rd))$, which is strongly Markov, see \cite[Th.1.1, Cor.4.1]{Ku20}. 

The uniqueness of the process in the case of unbounded coefficients is harder to prove. Up to our knowledge, at the moment  there are no  general results establishing the uniqueness of the solution to the Martingale problem for $(L,C_0^\infty(\rd))$, if the kernel $\nu(x,du)$ is unbounded, unless one can rewrite $L$ (using the change of measure) in the form, usually used in L\'evy-driven SDEs, and check the conditions for the existence and uniqueness of the (strong or weak) solution to the respective SDE. We do not pursue this topic in the current article and confine ourselves with finding the conditions that guarantee  the ergodicity of \emph{some} Markovian selection, related to $(L,C_0^\infty(\rd))$. 

 It is a kind of general knowledge that ergodicity can come from the drift or diffusion parts of the generator.  We investigate the conditions on the L\'evy-type measure under which \eqref{Lyap2} is satisfied.  Our interest lies both in the drift- and noise-induced ergodicity. Further, we will show that in the second case, when the ergodic behaviour comes from the kernel  $\nu(x,\cdot)$ itself,  the  support of $\nu(x,\cdot)$ must depend on the direction of the vector $x$.

\section{Settings and results}\label{Set}

Denote by $|u|$ the norm of a vector $u\in \rd$;  by $\Bb(r)$ and $\Bb^c(r)$ the ball centred at $0$ with radius $r$ and its complement, respectively; by $\Aa(r,R)$ the annulus with the inner radius $r$ and the outer radius $R$. 

As a Lyapunov function we take
\begin{equation}\label{Vp}
	V (x)= 1+\phi^p(x), \quad p\in (0,1), \quad x\in \rd,
\end{equation}
where $\phi: \rd\to \real_+$, $\phi\in C^2(\rd)$, such that $\phi(x) =|x|$ for $|x|>1$ and $\phi(x)\leq |x|$  for $|x|\leq 1$.

We also denote by $e_x = \frac{x}{\|x\|}$ the unit vector of $x \in \rd$, and by $\gamma_{x,u}$ the cosine of the angle between unit vectors $e_x, e_u$:
\begin{equation*}\label{gxu}
	\gamma_{x,u}:= \cos \angle (e_x,e_u). 
\end{equation*}

The aim of the article is to find conditions ensuring the Lyapunov-type inequality \eqref{Lyap2} for $\LL$. Our strategy is to select its negative and positive components to compare their impact on the generator. For the sake of simplicity, we consider separately the ``drift'' and ``kernel'' parts:
\begin{gather*}\label{split1}
    \LL V(x) = \nabla V(x)\cdot \Big(\ell(x) + \int\limits_{\Aa(1, |x|)} u\nu(x,du) \Big) \, +\\
    + \int\limits_{\rd} \left(V(x+u)- V(x) - \nabla V(x) \cdot u\I_{|u|\leq |x|}\right) \nu(x,du) =:\notag \\
		=: \LL^{drift} V(x)+ \LL^{ ker} V(x). \notag
\end{gather*}

Consider the ``drift'' part. For $|x|$ big enough $V(x) = 1 + |x|^p$, so we can transform
\begin{equation*}
	\begin{split}
		\LL^{drift} V(x) 
		&= p |x|^{p-1}\Big(|\ell(x)|\gamma_{x,\ell(x)} + \int\limits_{\Aa(1, |x|)} \gamma_{x,u} |u|\,\nu(x,du)\Big).  
	\end{split}
\end{equation*}

Since our ultimate goal is to compare the speed of all generator components, we define the first ``growth'' function
\begin{equation*}
	\phi^{drift}(x) := |x|^{p-1}\max{\Big(|\ell(x)|,\int_{\Aa(1, |x|)} |u|\,\nu(x,du)\Big)}.
\end{equation*}
It is not hard to see that there exists $\zeta \in [-2,2]$, such that
\begin{equation}\label{drift}
	\LL^{drift} V(x) \leq -p\zeta \phi^{drift}(x).
\end{equation}

Additionally,  fix $\lambda \in (0,1)$ and split the ``kernel'' part $\LL^{ker}$ into the ``ball'' and  ``tail'' sub-components, based on the jump size $|u|$:
\begin{gather}\label{Lsplit}
    \LL^{ker}V(x)
		= \int\limits_{\Bb(|x|^\lambda)} \left(V(x+u) - V(x) - \nabla V(x) \cdot u \right) \nu(x,du) \, +\\
		+ \int\limits_{\Bb^c(|x|^\lambda)} \left(V(x+u) - V(x)  - \nabla V(x) u \I_{\Aa(|x|^\lambda, |x|)}\right)\nu(x,du) =:\notag\\
		=: \LL^{ball}V(x)+\LL^{tail}V(x).\notag
\end{gather}

Further, again according to the jump size $|u|$, split the ``tail'' part:
\begin{gather*}\label{tailsplit}
    \LL^{tail}V(x)
		= \int\limits_{\Aa(|x|^\lambda,|x|)} \left(V(x+u) - V(x) - \nabla V(x)\cdot u \right) \nu(x,du) \, +\\
		+ \int\limits_{\Bb^c(|x|)} \left(V(x+u) - V(x) \right)\nu(x,du) =:\notag\\
		=: \LL^{big}V(x)+\LL^{large}V(x). \notag
\end{gather*}

Now we are ready to introduce the rest of the ``growth'' functions:
\begin{equation*}\label{phiPM}
	\phi^{ball}_+(x) := |x|^{p-2} \int\limits_{\Bb(|x|^\lambda)}|u|^2 \, \nu(x,du), \qquad \phi^{ball}_-(x) :=  |x|^{p-2} \int\limits_{\Bb(|x|^\lambda)}|u|^2 \gamma_{ x,u}^2\, \nu(x,du),
\end{equation*}
and
\begin{equation*}\label{phiBL}
	\phi^{big}(x) :=\frac{p}{2}|x|^{p-1} \int\limits_{\Aa(|x|^{\lambda},|x|)} |u| \, \nu(x,du),
	\qquad \phi^{large}(x) := \int\limits_{\Bb^c(|x|)} |u|^p \, \nu(x,du).
\end{equation*}

Therefore, we arrive at the following result.
\begin{lemma}\label{lem1}
	For $V$ from \eqref{Vp} and $\zeta \in [-2,2]$:
	\begin{equation}\label{L30}
		\LL V(x)
		\leq  -p\zeta \phi^{drift}(x)+\frac{p}{4}(1+o(1))(\phi^{ball}_+(x)+(p-2)\phi^{ball}_-(x)) + \phi^{big}(x)+\phi^{large}(x), \quad   |x| \to \infty .
	\end{equation}
\end{lemma}

Inequality \eqref{L30} gives us a first vision on how the final condition may look. Indeed, consider the right-hand side of \eqref{L30}. The last two terms (which relate to ``tail'') are positive by construction. Consequently, the negativity of $\LL V$ can come from either the ``drift'' or ``ball'' components, or a combination thereof. According to this, we introduce several constants that illustrate the interconnection of the corresponding ``growth'' functions.

First, it can be seen from \eqref{L30} that $\phi^{ball}_+$ and $\phi^{ball}_-$ correspond to the positive and negative parts of ``ball' respectively thus, the sign of $\LL^{ball} V$ is fully determined by their relation. Define
\begin{equation*}\label{S}
	C^{ball}:= 	\limsup_{|x| \to \infty} \frac{\phi^{ball}_+(x)}{\phi^{ball}_-(x)}.
\end{equation*}
By definition $C^{ball}\geq 1$ and it has to be less than $2-p$ for $\LL^{ball}V$ to be negative. In other words, it means that the measure $\nu(x,\cdot)$ has to be concentrated along the direction, collinear to $x$.

To compare ``drift'' with ``tail'' we introduce
\begin{equation*}
	C^{big}:= \limsup_{|x| \to \infty} \frac{\phi^{big}(x)}{\phi^{drift}(x)} \quad \text{and} \quad
	C^{large}:=  \limsup_{|x| \to \infty} \frac{\phi^{large}(x)}{\phi^{drift}(x)}. 
\end{equation*} 
In the same vein, the constant
\begin{equation}\label{Sb}
	C^{tail} := \limsup_{|x|\to\infty} \frac{\max{(\phi^{big}(x), \, \phi^{large}(x))}}{\phi^{ball}_-(x)}.
\end{equation}
shows the relation between ``ball'' and ``tail''.

In both cases, finite constants ensure that ``tail'' does not spoil the negativity of $\LL V$, and under additional assumptions one can get either ``drift''- or ``ball''-induced ergodicity (respectively). 
\begin{remark}
	In \eqref{Sb} we do not consider the constant, related to $\phi^{ball}_+$, because in the case of ``ball''-induced ergodicity, the strictly negative component of ``ball'' (that is $\phi^{ball}_-$) has to dominate.
\end{remark}

Finally, to compare ``drift'' and ``ball'', we define 
\begin{equation*}
	\tilde{C}^{drift}_+:= \limsup_{|x|\to\infty} \frac{\phi^{ball}_+(x)}{\phi^{drift}(x)} \quad \text{and} \quad
	\tilde{C}^{drift}_-:= \liminf_{|x|\to \infty} \frac{\phi^{ball}_-(x)}{\phi^{drift}(x)}.
\end{equation*}

Note that the functions $\phi^{drift}$ and $\phi^{ball}_-$ are not necessarily monotone; thus, in order to formulate our main result in a more ``user-friendly'' manner, we assume that they can be bounded from below by some power-type functions. Namely, suppose that for $|x|$ big enough there exist some strictly positive $K^{drift}, K^{ball}$, and $\gamma^{drift}$, $\gamma^{ball}$, such that the following inequalities hold:
\begin{equation}\label{gamdriftball}
	\phi^{drift}(x)\geq  K^{drift} |x|^{\gamma^{drift}}, \qquad
	\phi^{ball}_{-}(x)\geq  K^{ball} |x|^{\gamma^{ball}}.
\end{equation}

Now we are ready to formulate the main results of the paper.

\begin{theorem}\label{T2}
	Case 1. (\textit{``Drift-induced'' ergodicity}). Suppose that  \eqref{drift}  holds with  $\zeta \in (0,2]$, and $\tilde{C}^{drift}_+$, $\tilde{C}^{drift}_-$, $C^{big}$, $C^{large}$ are finite. Then \eqref{Lyap2} holds with $f(u) =   Ku^{\gamma^{drift}/p}, K > 0$, if
	\begin{equation}\label{LI1}
		- p\zeta + \frac{p}{4}\left(\tilde{C}^{drift}_+ +(p-2)\tilde{C}^{drift}_-\right) + C^{big}+C^{large} < 0.
	\end{equation}			
	
	Case 2. (\textit{``Ball-induced'' ergodicity}). Suppose that $C^{ball}, C^{tail} < \infty$. Then \eqref{Lyap2} holds with $f(u) =  Ku^{\gamma^{ball}/p}, K \!>\!0$, if either
	
	a) $\tilde{C}^{drift}_- = \infty$, $\zeta \in [-2,2]$ and
	\begin{equation}\label{LI2a}
		\frac{p}{4}(C^{ball} + p-2)+ 2C^{tail}<0.
	\end{equation}	
	or
	
	b) $\tilde{C}^{drift}_- \in (0,\infty)$, $\zeta \in [-2,0]$ and
	\begin{equation}\label{LI2b}
		-\frac{p\zeta}{\tilde{C}^{drift}_-} + \frac{p}{4}(C^{ball} + p-2)+ 2C^{tail}<0.
	\end{equation}	
\end{theorem}

\begin{remark}
	In Case 2 we have purely ``ball''-induced ergodicity: in Case 2a) ``drift'' does not have any impact comparing to ``ball'' and in Case 2b) ``drift'' is positive and thus works even against. In Case 1 ergodicity comes from ``drift'' only, if $\tilde{C}^{drift}_- = 0$; otherwise both components contribute. 
\end{remark}

\section{Proofs}\label{Proofs}

\begin{proof}[Proof of Lemma \ref{lem1}]
	We already have the estimate of $\LL^{drift} V$ (see \eqref{drift}).  Consider now $\LL^{ball} V$. Since $|u| \in \Bb(|x|^\lambda)$, then for $|x| \gg 1$:
	\begin{equation*}
		|x+u| \geq |x|\cdot |1-|x|^{\lambda-1}| >1,
	\end{equation*}
and $V(x+u)=1+|x+u|^p$. Thus, we can apply Taylor's formula to the integrand function  of the first term of \eqref{Lsplit}  and get
\begin{equation*}\label{TF}
		V(x+u) - V(x) - \nabla V(x)\cdot u = \frac12 \int_0^1 (1-\theta) u'D^2V(x+\theta u)ud\theta ,
\end{equation*}
	where
	\begin{equation*}\label{DV}
		D^2V(x) = \left\{\partial_{ij}^2 V(x)\right\}_{i,j=1}^d=p|x|^{p-2}
		\begin{pmatrix}
			1+(p-2)\frac{x_1^2}{|x|^2} & (p-2)\frac{x_1 x_2}{|x|^2} & \cdots& (p-2)\frac{x_1 x_d}{|x|^2} \\ 
			(p-2)\frac{x_1 x_2}{|x|^2} & 1+(p-2)\frac{x_2^2}{|x|^2} & \cdots & (p-2)\frac{x_2 x_d}{|x|^2}\\
			\cdots & \cdots & \ddots & \cdots\\
			(p-2)\frac{x_1 x_d}{|x|^2} & (p-2)\frac{x_2 x_d}{|x|^2} & \cdots &  1+(p-2)\frac{x_d^2}{|x|^2}
		\end{pmatrix}.
	\end{equation*}
	Then
    \begin{gather*}
        \LL^{ball}V(x)
		= \frac12  \int\limits_{\Bb(|x|^\lambda)} \int\limits_0^1  (1-\theta) u' \cdot D^2V(x+\theta u) \cdot u \nu(x,du)d\theta =\\
		=\frac{p}{2} \int\limits_{\Bb(|x|^\lambda)}  \int\limits_0^1 (1-\theta)|x+\theta u|^{p-2}\left(\sum_{i=1}^{d}{\left(1+(p-2)\frac{(x_i+\theta u_i)^2}{|x+\theta u|^2}\right)u_i^2}\right.\,+ \notag\\
		\left.+(p-2)\sum_{i,j=1, i\neq j}^{d}{u_i u_j\frac{(x_i+ \theta u_i) (x_j+\theta u_j)}{|x+\theta u|^2}}\right)d\theta\nu(x, du).\notag
    \end{gather*}	
    
	Regroup the sums in the inner integral and apply the sum of squares formula to $u_i (x_i+\theta u_i)$:
    \begin{gather*}
        \LL^{ball}V(x) 
		=\frac{p}{2}  \int\limits_{\Bb(|x|^\lambda)}  \int\limits_0^1 (1-\theta)|x+\theta u|^{p-2} {\left(\sum_{i=1}^{d}{u_i^2}+\frac{p-2}{|x+\theta u|^2}\left(\sum_{i=1}^{d}{u_i (x_i+\theta u_i)}\right)^2\right)}d\theta\nu(x, du) =\\
		=\frac{p}{2}  \int\limits_{\Bb(|x|^\lambda)}  \int\limits_0^1 (1-\theta)|x+\theta u|^{p-2}|u|^2\Big(1+\underbrace{(p-2)\left(e_u\cdot e_{x+\theta u}\right)^2}_{\text{negative}}\Big)d\theta\nu(x, du). \notag
    \end{gather*}

	Consider the squared scalar product
    \begin{gather*}\label{dotp}
        (e_{u}\cdot e_{x+\theta u} )^2
			=  \left( \gamma_{x,u}\frac{|x|}{|x+\theta u|} +   \frac{\theta |u|}{|x+\theta u|} \right)^2
			\geq \gamma_{x,u}^2 \left(\frac{|x|}{|x|+ |u|}\right)^2 - \frac{2\theta |x| \cdot|u|}{(|x|-|u|)^2} \geq \\
            \geq 
			\frac{\gamma_{x,u}^2 }{(1+|x|^{\lambda-1})^2}- 
			\frac{2 |x|^\lambda}{|x|(1-|x|^{\lambda-1})^2}. \notag
    \end{gather*}
	
	Estimating the inner integral of $\LL^{ball}V$ depending upon its sign (again by triangle inequalities), we get
	\begin{equation*}\label{triangle}
		\frac{1}{2} (|x|\cdot|1+|x|^{\lambda-1}|)^{p-2} \leq \int\limits_0^1 (1-\theta)|x+\theta u|^{p-2}\, d\theta \leq \frac12 (|x|\cdot|1-|x|^{\lambda-1}|)^{p-2}.
	\end{equation*}

	Summarising, we derive the final estimate for the ``ball'' part:
    \begin{gather*}
        \LL^{ball}V(x)
		\leq \frac{p}{4} |x|^{p-2} \int\limits_{\Bb(|x|^\lambda)}|u|^2 \left(C_1(x)+ (p-2)C_2(x) \gamma^2_{x, u}\right)\nu(x,du) = \\
		= \frac{p}{4} C_1(x) \phi^{ball}_+(x) +  \frac{p(p-2)}{4} C_2(x) \phi^{ball}_-(x),
    \end{gather*}
	where  
	\begin{equation*}\label{C12}
		C_1(x) =\frac{1}{\left(1-|x|^{\lambda-1}\right)^{2-p}}  +  \frac{2(2-p) |x|^{\lambda-1}}{(1+|x|^{\lambda -1})^{2-p}(1-|x|^{\lambda -1})^2}, \quad 
		C_2(x) = \frac{1}{\left(1+ |x|^{\lambda-1}\right)^{4-p}}.
	\end{equation*} 
	Both $C_1(x)$ and $C_2(x)$ tend to 1 as $|x|$ tends to infinity; therefore for $|x|$ large enough:
	\begin{equation}\label{Lball}
		\LL^{ball} V(x)
		\leq  \frac{p}{4}(1+o(1))(\phi^{ball}_+(x)+(p-2)\phi^{ball}_-(x)). 
	\end{equation}

	Next, consider the ``tail'' part, beginning with the big jumps. Since by definition $V(x)$ is either $1+ |x|^p$ or $1 + \phi^p(x)$ (depending on the value of $|x|$), and $\phi(x) \leq |x|$, then $V(x+u) \leq 1+ |x+u|^p$, and
    \begin{gather*}
        \LL^{big} V(x)
		\leq \int\limits_{\Aa(|x|^\lambda, |x|)} (|x+u|^p-|x|^p-p|x|^{p-1}\cdot |u|\gamma_{ x,u})\, \nu(x,du) = \\
		= |x|^p \int\limits_{\Aa(|x|^\lambda, |x|)} \left(\left(1 + 2 \frac{|u|}{|x|}\gamma_{x,u} + \left(\frac{|u|}{|x|} \right)^2\right)^{\frac{p}{2}}  -1\right)\, \nu(x,du) \,-\\
		- \,p|x|^{p-1} \int\limits_{\Aa(|x|^\lambda, |x|)}|u|\gamma_{ x,u}\, \nu(x,du) =: J_1(x)+J_2(x).
    \end{gather*}
	
	Now using the fact that $u \in \Aa(|x|^\lambda, |x|)$, we derive
    \begin{gather*}
        J_1(x)
			\leq |x|^p \int\limits_{\Aa(|x|^\lambda, |x|)} \Big(\Big(1 + (1+2\gamma_{ x,u})\frac{|u|}{|x|} \Big)^{\frac{p}{2}}  -1\Big) \nu(x,du) \leq \\
			\leq \frac{p}{2}|x|^{p-1} \int\limits_{\Aa(|x|^\lambda, |x|)} (1+2\gamma_{ x,u})|u|\,\nu(x,du).
    \end{gather*}

	Summing $J_1(x)$ and $J_2(x)$, we obtain the final estimate of the big jumps part:
	\begin{equation}\label{big}
		\LL^{big} V(x) \leq \frac{p}{2}|x|^{p-1} \int\limits_{\Aa(|x|^\lambda, |x|)} |u| \nu(x,du) = \phi^{big}(x). 
	\end{equation}

	In the case of large jumps we again use that by definition $V(x+u) \leq 1 + |x+u|^p$. Then
	\begin{equation}\label{large}
		\LL^{large}V(x) \le \int\limits_{\Bb^c(|x|)} \left(|x+u|^p-|x|^p\right) \nu(x,du) \le \int\limits_{\Bb^c(|x|)} |u|^p \, \nu(x,du) = \phi^{large}(x).
	\end{equation}

	Together estimates \eqref{drift}, \eqref{Lball}, \eqref{big} and \eqref{large} yield the statement of the lemma, namely,
	\begin{equation*}
		\LL V(x)
		\leq  -p\zeta \phi^{drift}(x)+\frac{p}{4}(1+o(1))(\phi^{ball}_+(x)+(p-2)\phi^{ball}_-(x)) + \phi^{big}(x) + \phi^{large}(x), \,  |x| \to \infty.
	\end{equation*} \qedhere
\end{proof}

	Now, using the results of the proven lemma, we can easily prove Theorem \ref{T2}. 

	\begin{proof}[Proof of Theorem~\ref{T2}] 
		\textit{Case 1. ``Drift-induced'' ergodicity}. Suppose  that  \eqref{drift}  holds with  $\zeta \in (0,2]$. Dividing both sides of \eqref{L30} by $\phi^{drift}$ and taking $\limsup$ as $|x|\to \infty$, we get
		\begin{equation}\label{Ld}
			\limsup_{|x| \to \infty} \frac{\LL V(x)}{\phi^{drift}(x)}	 \leq - p\zeta + \frac{p}{4}\left(\tilde{C}^{drift}_+ +(p-2)\tilde{C}^{drift}_-\right)+ C^{big} + C^{large}.
		\end{equation}
		Then it is clear that the Lyapunov-type condition \eqref{Lyap2} is satisfied if $\tilde{C}^{drift}_+$, $\tilde{C}^{drift}_-$, $C^{big}$, $C^{large}$ are finite, and moreover the right-hand side of \eqref{Ld} is strictly negative.
		
		\textit{Case 2. ``Ball-induced'' ergodicity}. 
        
		 a) Assume that $\tilde{C}^{drift}_- = \infty$ (that is, ``ball'' is way faster than ``drift'') and $\zeta \in [-2,2]$. Dividing this time both sides of \eqref{L30} by $\phi^{ball}_-$ and taking $\limsup$ as $|x|\to \infty$, we get
		\begin{equation}\label{Ld1}
			\limsup_{|x| \to \infty} \frac{\LL V(x)}{\phi^{ball}_-(x)}	 \leq 	\frac{p}{4}(C^{ball} + p-2)+ 2C^{tail}.
		\end{equation}
		If $C^{ball}, C^{tail} < \infty$, and the right-hand side of \eqref{Ld1} is strictly negative, condition \eqref{Lyap2} is satisfied.
	
		b) Assume $\tilde{C}^{drift}_- \in (0,\infty)$ and $\zeta \in [-2,0]$. Following the same approach as above we get
		\begin{equation}\label{Ld2}
			\limsup_{|x| \to \infty} \frac{\LL V(x)}{\phi^{ball}_-(x)}	 \leq 	-\frac{p\zeta}{\tilde{C}^{drift}_-} + \frac{p}{4}(C^{ball} + p-2)+ 2C^{tail}.
		\end{equation}
		If $C^{ball}, C^{tail} < \infty$, and the right-hand side of \eqref{Ld2} is strictly negative, condition \eqref{Lyap2} holds.
		
		In each case, the form of $f$ from \eqref{Lyap2} (and thus of the ergodic rates $\psi$ from \eqref{psi}) is determined by \eqref{gamdriftball}.
		
	\end{proof}

\section{Examples}\label{Examples}

\begin{example}
	Consider a one-dimensional symmetric bounded $\nu(\cdot, \cdot)$ and suppose that the drift part $\ell(x)$ is absent. Then $\phi^{drift}$, $\phi^{big}$ and $\phi^{large}$ are all equal to zero for all $x  \in \real$,  and we are in the setting of Case 2a) with $C^{tail}=0$. Moreover, given that we are in the one-dimensional setting, $|\gamma_{ x,u}|=1$, thus, $\phi^{ball}_+ = \phi^{ball}_-$ and $C^{ball}=1$. Then \eqref{LI2a} is always satisfied, yielding the Lyapunov-type condition \eqref{Lyap2}.
\end{example}

\begin{example}

Consider now an $\alpha$-stable Lévy-type  measure $\nu(x,du) = |x|^\beta |u|^{-d-\alpha} du, \,  x\in \rd,$ with $\beta \in (1-p,1)$ and $\alpha \in (0,2)$. Take $\ell(x) = |x|^\eta$, $\eta\in (0, \beta)$, and assume that $\zeta \in [1/2, 1]$.

First, estimate 
\begin{equation*}
	\phi^{ball}_-(x) \leq \phi^{ball}_+(x) = \frac{S(d)}{2-\alpha} |x|^{p-2 + \beta +\lambda(2-\alpha)},	
\end{equation*}
where $S(d)$ is the volume of a unit sphere in $\rd$.

Next, to calculate the rest of the ``growth'' functions we define
\begin{equation*}\label{phiAnn}
	\phi^{ann}_\tau(x)
	:= |x|^{p-1+\beta} \int\limits_{\Aa(|x|^\tau,|x|)} |u|^{1-d-\alpha} \,du, \quad \tau \in \{0,\lambda\}. 
\end{equation*} 
Simple calculations lead us to
\begin{equation*}
	\phi^{ann}_\tau(x)   \asymp  S(d)
	\begin{cases} 
		(\alpha-1)^{-1}|x|^{ p-1+\beta+ (1-\alpha)\tau}, & \alpha>1,\\
		(1-\alpha)^{-1}|x|^{ p+\beta-\alpha}, & \alpha<1, \\
		(1-\tau)|x|^{ p-1 + \beta }\ln |x|,	& \alpha =1. 
	\end{cases}
\end{equation*} 
Then (remember that $\eta \in (0,\beta)$)
\begin{equation*}
	\phi^{drift}(x)  = \phi^{ann}_0(x), \qquad \phi^{big}(x)  = \frac{p}{2}\phi^{ann}_\lambda(x).
\end{equation*}

Finally,
$$
\phi^{large}(x) = \frac{S(d)}{\alpha - p} |x|^{p+\beta-\alpha}. 
$$

We can easily see that $\tilde{C}^{drift}_+$ (hence and $\tilde{C}^{drift}_-$) are equal to zero for every $\alpha \in (0,2)$. Furthermore, given that $\zeta \in [1/2,1]$, inequality \eqref{LI1} is automatically satisfied for $\alpha \ge 1$, thus, Lyapunov-type condition \eqref{Lyap2} holds with $f(u) = Cu^{\frac{p-1+\beta}{p}}$, $C \leq p\zeta$, for $\alpha > 1$, and with $f(u) = Cu^{\frac{p-1+\beta}{p}}\ln{u}$, $C \leq \frac{p}{2}(2\zeta-1+\lambda)$, for $\alpha = 1$. For $\alpha<1$, \eqref{Lyap2} holds with $f(u) = Cu^{\frac{p+\beta-\alpha}{p}}$, $C \leq p\left(\zeta-\frac{1}{2}\right)-\frac{1-\alpha}{\alpha-p}$, if $p$ satisfies the inequality
$$
	p^2(1-2\zeta)-p\alpha(1-2\zeta)-2(1-\alpha)>0.
$$
In each case, $C$ was derived directly from \eqref{LI1}.

\end{example}

\section*{Aknowledgement.} 
This research was partly funded by the ``DAAD Ostpartnerschaften Programme'' of the TU Dresden. The author is grateful to Faculty of Mathematics at TU Dresden and Prof.\ Ren\'e Schilling for providing excellent working conditions.

The author expresses deepest gratitude to her supervisor, Dr. Victoria Knopova, for her mentorship and constant encouragement.

The author also sincerely thanks Prof.\ Ren\'e Schilling for his valuable and insightful comments and suggestions.

\bibliographystyle{abbrv}

\bibliography{GeneralArxiv}

\end{document}